\documentclass[reqno,10pt]{amsart}
\usepackage{amsmath,amssymb,amsthm,mathrsfs}
\usepackage{longtable}
\usepackage{enumerate}
\usepackage{graphicx}
\usepackage{subcaption}
\usepackage{wasysym}
\usepackage{upgreek}
\usepackage[dvipsnames]{xcolor}
\usepackage{cancel}
\usepackage{soul}
\usepackage{mathtools}
\usepackage{comment}
\usepackage{tikzit,tikz}
\usetikzlibrary{arrows,automata,positioning}

\tikzstyle{burbuja}=[fill=white, draw=black, shape=circle]

\tikzstyle{flecha}=[->,>=stealth,
thick,
shorten <=2pt,
shorten >=2pt,]

\usetikzlibrary{arrows}

\allowdisplaybreaks

\theoremstyle{definition}
\newtheorem{theorem}{Theorem}[section]
\newtheorem{lemma}[theorem]{Lemma}
\newtheorem{prop}[theorem]{Proposition}
\newtheorem{corollary}[theorem]{Corollary}
\newtheorem{definition}[theorem]{Definition}

\usepackage{color}

\title{Ladder operators for generalized Zernike or disk polynomials}
\author{Misael E. Marriaga}
\date{\today}

\address[M. E. Marriaga]{Departamento de Matem\'atica Aplicada, Ciencia e Ingenier\'ia de Materiales y
	Tecnolog\'ia Electr\'onica, Universidad Rey Juan Carlos (Spain)}
\email{misael.marriaga@urjc.es}

\thanks{MEM has been supported by the research project [PID2021- 122154NB-I00], \emph{Ortogonalidad y Aproximación con Aplicaciones en Machine Learning y Teoría de la Probabilidad}  funded  by MICIU/AEI/10.13039/501100011033 and by ``ERDF A Way of making Europe”. The author has also been supported by the Comunidad de Madrid multiannual agreement with the Universidad Rey Juan Carlos under the grant Proyectos I+D para Jóvenes Doctores, Ref. M2731, project NETA-MM}

\begin{document}
	
	\begin{abstract}
		The aim of this work is to report on several ladder operators for generalized Zernike polynomials which are orthogonal polynomials on the unit disk $\mathbf{D}\,=\,\{(x,y)\in \mathbb{R}^2: \; x^2+y^2\leqslant 1\}$ with respect to the weight function $W_{\mu}(x,y)\,=\,(1-x^2-y^2)^{\mu}$ where $\mu>-1$. These polynomials can be expressed in terms of the univariate Jacobi polynomials and, thus, we start by deducing several ladder operators for the Jacobi polynomials. Due to the symmetry of the disk and the weight function $W_{\mu}$, it turns out that it is more convenient to use complex variables $z\,=\, x+iy$ and $\bar{z}\,=\,x-iy$. Indeed, this allows us to systematically use the univariate ladder operators to deduce analogous ones for the complex generalized Zernike polynomials. Some of these univariate and bivariate ladder operators already appear in the literature. However, to the best of our knowledge, the proofs presented here are new. Lastly, we illustrate the use of ladder operators in the study of the orthogonal structure of some Sobolev spaces.
	\end{abstract}

\maketitle

\section{Introduction}

The so-called Zernike polynomials, originally introduced by Frits Zernike in 1934 (\cite{Ze34}) to describe the diffraction of the wavefront in the phase contrast image microscope, are used to characterize higher-order errors observed in interferometric analysis in precision optical manufacturing to achieve the desired performance of the system. In ophthalmological practice, the Hartmann-Shack sensor (or wavefront sensor) is used to determine the refractive errors of the human optical system, measuring slopes or normals of the wavefront at different points starting from the displacement of some luminous points in a target. A systematic method of classifying forms of aberration is to express the corresponding function in an appropriate basis. Zernike polynomials are recognized as the standard basis of wavefront developments by the Optical Society of America, (OSA). Any sufficiently regular phase function defined on the unit disk can be represented by its Fourier expansion in terms of the Zernike polynomials with certain coefficients. The alteration of these coefficients allows detection of the possible aberrations of the studied optical system. They are also used to describe the aberrations of the cornea or lens from an ideal spherical shape in optometry and ophthalmology. Finally, they can be effectively used in adaptive optics to cancel
atmospheric distortion, allowing images to be improved in IR or visual astronomy and satellite
images.

In \cite{T08}, an application in quantum optics of the “generalized Zernike or disk polynomials” (see, for instance, \cite{W05}) is shown, resorting to the Lie algebra based investigation of the dynamics of quantum systems driven by two-mode interaction Hamiltonians. This leads to \textit{su}(2) and
\textit{su}(1,1) Lie algebraic structures, whose generators, in fact, can appropriately be realized in terms of creation and annihilation operators through raising and lowering operators. Hence, studying raising and lowering operators (collectively known as ladder operators), which is one the main goals of this paper, is of much interest in the theory of quantum optics.

From our point of view, Zernike polynomials are polynomials in two variables which are orthogonal on the unit disk with respect to the Lebesgue measure. They are represented in polar coordinates as a product of a radial polynomial part times a trigonometric function. The even polynomials are multiples of the cosine, and the odd polynomials are multiples of the sine. Generalized Zernike polynomials are bivariate polynomials which are orthogonal in the unit disk $x^2+y^2\leqslant 1$ with respect to the weight function $(1-x^2-y^2)^{\mu}$, where $\mu>-1$ is a parameter and the special case $\mu=0$ is equivalent to the usual Zernike polynomials. In turn, these polynomials can be generalized to several variables represented as a product of univariate Jacobi polynomials shifted to the interval $[0,1]$ and harmonic polynomials (see \cite{DX14}).

Another way of studying the generalized Zernike polynomials is to express them in one complex variable, for which we identify $\mathbb{R}^2$ with $\mathbb{C}$ by setting $z=x+iy$ and consider the unit disk as a subset of $\mathbb{C}$. Expressing orthogonal polynomials in a complex variable can be more convenient, even essential, in some cases, and may result in more elegant formulas and relations as well as simpler proofs. It is worth noting that the use of complex variables instead of cartesian ones to achieve results connected with orthogonal polynomials is discussed in \cite{Xu15}.

In this paper, we use the complex version of the generalized Zernike polynomials to show that they satisfy several ladder operators. These polynomials are represented in terms of univariate Jacobi polynomials and powers of the complex variable. We take advantage of this fact to deduce the ladder operators presented in the sequel and, thus, we also study several ladder operators satisfied by the Jacobi polynomials. We also include a brief discussion about the relevance of ladder operators in the study of orthogonal polynomials associated with Sobolev inner products, which are inner products involving the derivatives of the polynomials.

The outline of this paper is as follows. In Section \ref{preliminaries}, we provide the basic background on bivariate orthogonal polynomials and the generalized Zernike polynomials. Section \ref{ladderJacobi} is concerned with the ladder operators for Jacobi polynomials. We note that some of these univariate ladder operators already appear in the literature. However, to the best of our knowledge, the proofs presented here are new. We collect the ladder operators for the generalized Zernike polynomials in Section \ref{ladderZernike} and, in Section \ref{recurrence}, we show that fundamental differential and recurrence relations for the generalized Zernike polynomials can be deduced from the ladder operators. We dedicate the last section to discuss how ladder operators can be used to study the orthogonal polynomials associated with two Sobolev inner products defined on the unit disk. These Sobolev orthognal polynomials have been studied previously without involving bivariate ladder operators. It turns out that our proofs are simpler, which justifies this brief sidenote.

\section{Generalized Zernike or disk polynomials}\label{preliminaries}

In this section, we recall the basic facts about bivariate orthogonal polynomials and the generalized Zernike polynomials that we will need in the sequel. Our main reference is \cite{DX14}.

\subsection{Bivariate orthogonal polynomials.}
We denote by $\Pi^2$ the linear space of real bivariate polynomials. For $n\geqslant 0$, let $\Pi^2_n$ denote the linear space of real bivariate polynomials of total degree at most $n$. Evidentely,
$$
\text{dim}\, \Pi_n^2\,=\,\binom{n+2}{2} \quad \text{and} \quad \Pi^2 \,=\,\bigcup_{n\geqslant 0} \Pi_n^2.
$$
We say that a sequence $\mathcal{P}\,=\,\{P_{n,m}(x,y):\,n\geqslant 0,\, 0\leqslant m \leqslant n\}$ of polynomials in $\Pi^2$ is a \textit{polynomial system (PS)} if for all $n\geqslant 0$, the set $\mathcal{P}_n\,=\{P_{n,m}(x,y): 0\leqslant m \leqslant n\}$ consists of $n+1$ linearly independent polynomials of total degree $n$, that is, $\deg P_{n,m}\,=\,n$, $0\leqslant m \leqslant n$. In this way, a PS $\mathcal{P}$ is a basis of $\Pi^2$.

Let $\langle \cdot, \cdot \rangle:\Pi^2\times \Pi^2\rightarrow \mathbb{R}$ be an inner product defined on polynomials. A polynomial $P$ of degree $n$ is called an \textit{orthogonal polynomial} with respect to the inner product if
$$
\langle P, Q \rangle \,=\,0, \quad \forall Q\in \Pi^2_{n-1}.
$$

Given an inner product $\langle \cdot, \cdot\rangle$ and a PS $\mathcal{P}\,=\,\{P_{n,m}(x,y):\,n\geqslant 0,\, 0\leqslant m \leqslant n\}$, we will say that $\mathcal{P}$ is \textit{orthogonal with respect to the inner product} if 
$$
\langle P_{n,m}, Q\rangle \,=\,0, \quad \forall Q\in \Pi^2_{n-1},
$$
for all $n\geqslant 0$ and $0\leqslant m \leqslant n$. Moreover, if 
$$
\langle P_{n,m}, P_{i,j}\rangle \,=\,H_{n,m}\,\delta_{n,i}\,\delta_{m,j}, 
$$
where $H_{n,m}\,\ne \,0$ for $n\geqslant 0$, then we say that $\mathcal{P}$ is a \textit{mutually orthogonal polynomial system}. Here, $\delta_{n,k}$ denotes the Kronecker delta.

Let $W(x,y)$ be a weight function defined on a domain $\Omega\subseteq \mathbb{R}^2$. That is, $W(x,y)>0$ for all $(x,y)\in \Omega$, and $\Omega$ has a nonempty interior. If an inner product is given by
$$
\langle P, Q\rangle \,=\, \int \int_{\Omega} P(x,y)\,Q(x,y)\,W(x,y)\,dxdy, \quad \forall P,Q\in \Pi^2,
$$
we say that the orthogonal polynomials, whenever they exist, are orthogonal with respect to the weight function $W$.

\subsection{Generalized Zernike polynomials}

Let 
$$
\mathbf{D}\,=\,\{(x,y)\in \mathbb{R}^2: \; x^2+y^2\leqslant 1\}
$$
denote the unit disk in $\mathbb{R}^2$. For $\mu>-1$, we define the weight function
$$
W_{\mu}(x,y)\,=\,(1-x^2-y^2)^{\mu}, \quad (x,y)\in \mathbf{D},
$$
and the inner product
$$
\langle P,Q \rangle_{\mu}\,=\,b_{\mu}\,\int\int_{\mathbf{D}}P(x,y)\,Q(x,y)\,W_{\mu}(x,y)\,dxdy,
$$
where
$$
b_{\mu}\,=\,\left(\int\int_{\mathbf{D}}W_{\mu}(x,y)\,dxdy \right)^{-1}\,=\,\frac{\mu+1}{\pi}.
$$

A mutually orthogonal polynomial system on the unit disk can be given in terms of the Jacobi polynomials in the polar coordinates $(x,y)\,=\,(r\cos\theta,r\sin\theta)$, $0\leqslant r\leqslant 1$ and $0\leqslant \theta\leqslant 2\pi$. Recall that the Jacobi polynomial of degree $n$ is given by
$$
P_n^{(\alpha,\beta)}(t)=\frac{1}{n!}\displaystyle\sum_{k=0}^n\binom{n}{k}(k+\alpha+1)_{n-k}(n+\alpha+\beta+1)_k\left(\dfrac{t-1}{2} \right)^k, 
$$
where, as usual,
$$
(a)_{0}=1, \quad (a)_{k} = a\,(a+1)\cdots(a+k-1), \quad k\geqslant 1,
$$
denotes the Pochhammer symbol. These polynomials are orthogonal with respect to the univariate weight function
$$
w_{\alpha,\beta}(t)\,=\,(1-t)^{\alpha}(1+t)^{\beta}, \quad \alpha,\beta>-1, \quad t\in[-1,1],
$$
and satisfy the differential equation
\begin{equation}\label{diffeq}
	L^{\alpha,\beta}[y]\equiv (1-t^2)y''+(\beta-\alpha-(\alpha+\beta+2)\,t)\,y'\,=\,-n\,(n+\alpha+\beta+1)\,y.
\end{equation}

\begin{prop}[\cite{DX14}]
	For $n\geqslant 0$ and $0\leqslant j \leqslant \frac{n}{2}$, define
\begin{equation}\label{zernikereal}
	\begin{aligned}
		&P^{n,\mu}_{j,1}(x,y)\,=\,P_j^{(\mu,n-2j)}(2r^2-1)\,r^{n-2j}\,\cos (n-2j)\theta,\\
		&P^{n,\mu}_{j,2}(x,y)\,=\,P_j^{(\mu,n-2j)}(2r^2-1)\,r^{n-2j}\,\sin (n-2j)\theta.
	\end{aligned}
\end{equation}
Then $\mathcal{P}^{\mu}\,=\,\{P^{n,\mu}_{j,1}(x,y),\,n\geqslant 0,\,0\leqslant j \leqslant \frac{n}{2}\}\cup \{P^{\mu}_{j,2}(x,y),\,n\geqslant 0,\,1\leqslant j \leqslant \frac{n}{2}\}$ constitutes a mutually orthogonal polynomial system with respect to $W_{\mu}(x,y)$.

Moreover,
$$
\langle P^{n,\mu}_{j,\nu},P^{m,\mu}_{k,\eta} \rangle_{\mu}\,=\,H^{n,\mu}_j\,\delta_{n,m}\,\delta_{j,k}\,\delta_{\nu,\eta},
$$
where
$$
H^{n,\mu}_j\,=\,\frac{(\mu+1)_j\,(n-j)!\,(n-j+\mu+1)}{j!\,(n+2)_{n-j}\,(n+\mu+1)}\left\{\begin{array}{ll}\times 2, & n\ne 2j,\\
\times 1, & n=2j. \end{array}\right.
$$
\end{prop}

\subsection{Complex generalized Zernike polynomials}

Due to the symmetry of the unit disk, expressing the generalized Zernike polynomials in complex variables can be more convinient for our study and results in more elegant formulas and relations. 

For $z\in \mathbb{C}$ we write $z=x+iy$ and, hence, considering the unit disk as a subset of $\mathbb{C}$, we have
$$
\mathbf{D}\,=\,\{z\in \mathbb{C}:\,z\bar{z}\leqslant 1\}.
$$

\begin{definition}
	For $k,j\geqslant 0$, we define
	\begin{equation}\label{zernikecomplex}
			Q_{k,j}^{\mu}(z,\bar{z})\,=\,\frac{j!}{(\mu+1)_j}\,z^{k-j}\,P_j^{(\mu,k-j)}(2z\bar{z}-1), \quad k>j.
	\end{equation}
\end{definition}

The polynomials \eqref{zernikecomplex} are normalized by $Q^{\mu}_{k,j}(1,1)\,=\,1$. We note that we can write (\cite{VK1993})
\begin{equation}\label{conjugation}
	Q^{\mu}_{k,j}(z,\bar{z})\,=\,\frac{k!}{(\mu+1)_k}\,\bar{z}^{j-k}\,P_k^{(\mu,j-k)}(2z\bar{z}-1).
\end{equation}
They constitute a mutually orthogonal system with respect to the weight function
$$
w_{\mu}(z)\,=\,(1-z\bar{z})^{\mu}, \quad \mu>-1,\quad z\in \mathbf{D},
$$
which satisfies $w_\mu(z)\,=\,W_{\mu}(x,y)$, hence, $w_{\mu}(z)\,dz\,=\,W_{\mu}(x,y)\,dxdy$. More precisely, the orthogonality is given by
\begin{equation}\label{orthogonality}
b_{\mu}\,\int_{\mathbf{D}}Q^{\mu}_{k,j}(z,\bar{z})\,\overline{Q^{\mu}_{m,\ell}(z,\bar{z})}\,w_{\mu}(z)\,dz\,=\,h^{\mu}_{k,j}\,\delta_{k,m}\,\delta_{j,\ell},
\end{equation}
where
\begin{equation}\label{norm}
h^{\mu}_{k,j}\,=\,\frac{\mu+1}{\mu+k+j+1}\frac{k!\,j!}{(\mu+1)_k\,(\mu+1)_j}.
\end{equation}
By \eqref{zernikecomplex}, 
\begin{equation*}
	\begin{aligned}
		\text{Re}\{Q^{\mu}_{n-j,j}(z,\bar{z})\}&\,=\,\frac{j!}{(\mu+1)_j}\,P^{n,\mu}_{j,1}(x,y), & 0\leqslant j \leqslant \frac{n}{2}, \\[10pt]
		\text{Im}\{Q^{\mu}_{n-j,j}(z,\bar{z})\}&\,=\,\frac{j!}{(\mu+1)_j}\,P^{n,\mu}_{j,2}(x,y), & 0\leqslant j \leqslant \frac{n}{2},
	\end{aligned}
\end{equation*}
which are, up to a constant, the orthogonal polynomials of two real variables defined in \eqref{zernikereal}.

We must note that the polynomials $Q_{k,j}^{\mu}(z,\bar{z})$ are invariant under the simultaneous permutations of the variables $\{z,\bar{z}\}$ and the subindices $\{k,j\}$, that is,
\begin{equation}\label{invariance}
Q_{k,j}^{\mu}(z,\bar{z})\,=\,Q_{j,k}^{\mu}(\bar{z},z).
\end{equation}
This invariance will permeate the results presented in the sequel.

\section{Ladder operators for Jacobi polynomials}\label{ladderJacobi}

Throughout this section, we will introduce pairs of related ladder operators for the univariate Jacobi polynomials and, in the next section, we use them to deduce ladder operators for the complex generalized Zernike polynomials. Some of these univariate and bivariate ladder operators already appear in the literature. However, to the best of our knowledge, the proofs presented here are new.

Define the operators
$$
A_1[u]\,=\,\frac{du}{dt} \quad \text{and} \quad A_2[u]\,=\,(1-t^2)\frac{du}{dt}+[\beta\,(1-t)-\alpha\,(1+t)]\,u.
$$
The action of these operators over the Jacobi polynomials is well known (\cite[18.9.15, 18.9.16]{NIST2010}). The classical proof for the following result is based on the orthogonality of the Jacobi polynomials. However, we provide an operational proof.

\begin{prop}
	The univariate Jacobi polynomials satisfy
	\begin{equation}\label{ladder1}
		\begin{aligned}
			A_1[P_n^{(\alpha,\beta)}(t)]&\,=\,\frac{1}{2}(n+\alpha+\beta+1)\,P_{n-1}^{(\alpha+1,\beta+1)}(t), & \alpha,\beta>-1,\\[10pt]
			A_2[P_n^{(\alpha,\beta)}(t)]&\,=\,-2(n+1)\,P_{n+1}^{(\alpha-1,\beta-1)}(t), & \alpha,\beta>0.
		\end{aligned}
	\end{equation}
\end{prop}
\begin{proof}
	Using the differential operator $L^{\alpha,\beta}$ defined in \eqref{diffeq}, we compute
	$$
	L^{\alpha,\beta}A_1[u]-A_1L^{\alpha,\beta}[u]\,=\,\left(2t\frac{d}{dt}+\alpha+\beta+2\right)A_1[u],
	$$
	therefore,
	$$
	L^{\alpha+1,\beta+1}A_1[u]\,=\,A_1\left[L^{\alpha,\beta}[u]+(\alpha+\beta+2)\,u\right].
	$$
	Letting $u=P_n^{(\alpha,\beta)}(t)$ and using \eqref{diffeq}, we get
	\begin{align*}
		L^{\alpha+1,\beta+1}A_1[P_n^{(\alpha,\beta)}(t)]&\,=\,A_1L^{\alpha,\beta}[P_n^{(\alpha,\beta)}(t)]+(\alpha+\beta+2)A_1[P_n^{(\alpha,\beta)}(t)]\\[10pt]
		&\,=\,\left(-n\,(n+\alpha+\beta+1)+\alpha+\beta+2 \right)\,A_1[P_n^{(\alpha,\beta)}(t)]\\[10pt]
		&\,=\,-(n-1)(n+\alpha+\beta+2)\,A_1[P_n^{(\alpha,\beta)}(t)],
	\end{align*}
	which implies that $A_1[P_n^{(\alpha,\beta)}(t)]\,=\,c_n^{\alpha,\beta}\,P_{n-1}^{(\alpha+1,\beta+1)}(t)$ for some constant $c_n^{\alpha,\beta}$. From the explicit expression of the Jacobi polynomials and comparing the leading coefficients on both sides of this last equation, we get
	$$
	c_n^{\alpha,\beta}\,=\,\frac{1}{2}(n+\alpha+\beta+1).
	$$
	
	Similarly,
	$$
	L^{\alpha,\beta}A_2[u]-A_2L^{\alpha,\beta}[u]\,=\,\left(2t\frac{d}{dt}-\alpha-\beta\right)A_2[u],
	$$
	or, equivalently,
	$$
	L^{\alpha-1,\beta-1}A_2[u]\,=\,A_2\left[L^{\alpha,\beta}[u]-(\alpha+\beta)\,u\right].
	$$
	Thus,
	\begin{align*}
		L^{\alpha-1,\beta-1}A_2[P_n^{(\alpha,\beta)}(t)]&\,=\,\left(-n\,(n+\alpha+\beta+1)-\alpha-\beta \right)\,A_2[P_n^{(\alpha,\beta)}(t)]\\[10pt]
		&\,=\,-(n+1)(n+\alpha+\beta)\,A_2[P_n^{(\alpha,\beta)}(t)],
	\end{align*}
	which implies that $A_2[P_n^{(\alpha,\beta)}(t)]\,=\,d_n^{\alpha,\beta}\,P_{n+1}^{(\alpha-1,\beta-1)}(t)$ for some constant $d_n^{\alpha,\beta}$. In order to deduce the value of $d_n^{\alpha,\beta}$, we note that we can write
	$$
	A_2[u]\,=\,-(t-1)^2\,\frac{du}{dt}-(\alpha+\beta)(t-1)\,u-2(t-1)\frac{du}{dt}-\alpha\,u.
	$$
	Using this last identity and the explicit expression of the Jacobi polynomials, we compare leading coefficients on both sides of $A_2[P_n^{(\alpha,\beta)}(t)]\,=\,d_n^{\alpha,\beta}\,P_{n+1}^{(\alpha-1,\beta-1)}(t)$ and deduce that $d_n^{\alpha,\beta}\,=\,-2\,(n+1)$.
\end{proof}


For $n\geqslant 0$, let
$$
B_1[u]\,=\,(1+t)\,\frac{du}{dt}+(n+\alpha+\beta+1)\,u \quad \text{and} \quad B_2[u]\,=\,(1-t^2)\frac{du}{dt}-[2\alpha+(1-t)\,n]\,u.
$$ 
We note that the coefficients of these operators are not fixed and depend on $n$.

\begin{prop}\label{univariateladders2}
	For $\beta>-1$, the univariate Jacobi polynomials satisfy
	\begin{equation}\label{ladder2}
		\begin{aligned}
			B_1[P_n^{(\alpha,\beta)}(t)]&\,=\,(n+\alpha+\beta+1)\,P_{n}^{(\alpha+1,\beta)}(t), & \alpha>-1,\\[10pt]
			B_2[P_n^{(\alpha,\beta)}(t)]&\,=\,-2(n+\alpha)\,P_{n}^{(\alpha-1,\beta)}(t), & \alpha>0.
		\end{aligned}
	\end{equation}
\end{prop}
\begin{proof}
	First, we note that we can write
	\begin{align*}
		L^{\alpha,\beta}[u]&\,=\,[2(1+t)-(t+1)^2]\frac{d^2u}{dt^2}+[2(\beta+1)-(\alpha+\beta+2)(1+t)]\frac{du}{dt}\\[10pt]
		&\,=\,2(1+t)\frac{d^2u}{dt^2}+2(\beta+1)\frac{du}{dt}-\left\{(1+t)\frac{d}{dt}-n\right\}B_1[u]-n(n+\alpha+\beta+1)\,u.
	\end{align*}
	From here, if we commute $L^{\alpha,\beta}$ with $B_1$, we get
	\begin{align*}
		L^{\alpha,\beta}B_1[u]-B_1L^{\alpha,\beta}[u]&\,=\,2(1+t)\frac{d^2u}{dt^2}+2(\beta+1)\frac{du}{dt}\\[10pt]
		&\,=\,\left\{(1+t)\frac{d}{dt}-n\right\}B_1[u]+L^{\alpha,\beta}[u]+n(n+\alpha+\beta+1)\,u,
	\end{align*}
	or, equivalently,
	$$
	L^{\alpha+1,\beta}B_1[u]\,=\,B_1\left(L^{\alpha,\beta}[u]-n\,u\right)+L^{\alpha,\beta}[u]+n(n+\alpha+\beta+1)\,u.
	$$
	From \eqref{diffeq}, it follows that
	$$
	L^{\alpha+1,\beta}B_1[P_n^{(\alpha,\beta)}(t)]\,=\,-n(n+\alpha+\beta+2)\,B_1[P_n^{(\alpha,\beta)}(t)],
	$$
	which implies that $B_1[P_n^{(\alpha,\beta)}(t)]=c_n^{\alpha,\beta}\,P_n^{(\alpha+1,\beta)}(t)$. Writing,
	$$
	B_1[u]\,=\,(t-1)\frac{du}{dt}+(n+\alpha+\beta+1)\,u+2\frac{du}{dt},
	$$
	we can compare leading coefficients and obtain
	$$
	c_n^{\alpha,\beta}\,=\,n+\alpha+\beta+1.
	$$
	
	Notice that we can write
	$$
	B_2[u]\,=\,-(t-1)^2\frac{du}{dt}+n\,(t-1)\,u-2(t-1)\frac{du}{dt}-2\,\alpha\,u.
	$$
	Hence, 
	$$
	B_2\left[\left(\frac{t-1}{2} \right)^n\right]\,=\,-2(n+\alpha)\,\left(\frac{t-1}{2} \right)^n,
	$$
	which implies that $B_2[P_n^{(\alpha,\beta)}(t)]$ is a polynomial of degree $n$. By the orthogonality of $\{P_{n}^{(\alpha-1,\beta)}(t)\}_{n\geqslant 0}$ and comparing leading coefficients, we have
	$$
	B_2[P_n^{(\alpha,\beta)}(t)]\,=\,-2(n+\alpha)P_{n}^{(\alpha-1,\beta)}(t)+\sum_{k=0}^{n-1}d_k^{\alpha,\beta}\,P_{k}^{(\alpha-1,\beta)}(t),
	$$
	where
	$$
	d_k^{\alpha,\beta}\,h_k^{\alpha-1,\beta}\,=\,\int_{-1}^1B_2[P_n^{(\alpha,\beta)}(t)]\,P_{k}^{(\alpha-1,\beta)}(t)\,w_{\alpha-1,\beta}(t)\,dt, \quad 0\leqslant k \leqslant n-1,
	$$
	with
	$$
	h_k^{\alpha-1,\beta}\,=\,\int_{-1}^1\left(P_{k}^{(\alpha-1,\beta)}(t)\right)^2\,w_{\alpha-1,\beta}(t)\,dt,
	$$
	which is non zero for $\alpha-1,\beta>-1$ (see \cite{Sz78}). After integrating by parts and observing that $p(t)w_{\alpha,\beta}(t)|_{-1}^{1}\,=\,0$ for every polynomial $p(t)$, we obtain 
	$$
			\int_{-1}^1B_2[P_n^{(\alpha,\beta)}(t)]\,P_{k}^{(\alpha-1,\beta)}(t)\,w_{\alpha-1,\beta}(t)\,dt\,=\,-\int_{-1}^1P_n^{(\alpha,\beta)}(t)\,B_1[P_{k}^{(\alpha-1,\beta)}(t)]\,w_{\alpha,\beta}(t)\,dt,
	$$
	and, therefore,
	$$
	d_k^{\alpha,\beta}\,h_k^{\alpha-1,\beta}\,=\,-\int_{-1}^1P_n^{(\alpha,\beta)}(t)\,B_1[P_{k}^{(\alpha-1,\beta)}(t)]\,w_{\alpha,\beta}(t)\,dt.
	$$
	But $B_1[P_{k}^{(\alpha-1,\beta)}(t)]$ is a polynomial of degree at most $k$. From the orthogonality of the Jacobi polynomials, we deduce that $d_k^{\alpha,\beta}\,=\,0$ for $0\leqslant k \leqslant n-1$, and \eqref{ladder2} is now proved.
\end{proof}

Notice that in \eqref{ladder2}, only the first parameter of the Jacobi polynomials is shifted by one unit. Recall that the univariate classical Jacobi polynomials following reflection formula 
\begin{equation}\label{reflection}
	P_n^{(\alpha,\beta)}(t)\,=\,(-1)^n\,P_n^{(\beta,\alpha)}(-t).
\end{equation}
Using \eqref{reflection}, it is straightforward to verify that, for $n\geqslant 0$, the operators defined as 
$$
C_1[u]\,=\,(1-t)\,\frac{du}{dt}-(n+\alpha+\beta+1)\,u \quad \text{and} \quad C_2[u]\,=\,(1-t^2)\frac{du}{dt}+[2\beta+(1+t)\,n]\,u,
$$
act by shifting only the second parameter.

\begin{prop}
	For $\alpha>-1$, the univariate Jacobi polynomials satisfy
	\begin{equation}\label{ladder3}
		\begin{aligned}
			C_1[P_n^{(\alpha,\beta)}(t)]&\,=\,-(n+\alpha+\beta+1)\,P_{n}^{(\alpha,\beta+1)}(t), & \beta>-1,\\[10pt]
			C_2[P_n^{(\alpha,\beta)}(t)]&\,=\,2(n+\beta)\,P_{n}^{(\alpha,\beta-1)}(t), & \beta>0.
		\end{aligned}
	\end{equation}
\end{prop}

On one hand, the operators in \eqref{ladder2} and \eqref{ladder3} shift one parameter and leave the degree of the polynomial unchanged. In contrast, the following univariate operators defined for $n\geqslant 0$ as
$$
D_1[u]\,=\,(1+t)\,\frac{du}{dt}-n\,u \quad \text{and} \quad D_2[u]\,=\,(1-t^2)\frac{du}{dt}+[(n+\beta+1)(1-t)-\alpha\,(1+t)]\,u,
$$
shift the parameter $\alpha$ as well as the degree of the polynomial.

\begin{prop}
	For $\beta>-1$, the univariate Jacobi polynomials satisfy
	\begin{equation}\label{ladder4}
		\begin{aligned}
			D_1[P_n^{(\alpha,\beta)}(t)]&\,=\,(n+\beta)\,P_{n-1}^{(\alpha+1,\beta)}(t), & \alpha>-1,\\[10pt]
			D_2[P_n^{(\alpha,\beta)}(t)]&\,=\,-2(n+1)\,P_{n+1}^{(\alpha-1,\beta)}(t), & \alpha>0.
		\end{aligned}
	\end{equation}
\end{prop}

\begin{proof}
	Observe that 
	$$
	D_1[u]\,=\,B_1[u]-(2n+\alpha+\beta+1)\,u.
	$$
	Therefore,
	\begin{align*}
		&L^{\alpha,\beta}D_1[u]-D_1L^{\alpha,\beta}[u]\,=\,L^{\alpha,\beta}B_1[u]-B_1L^{\alpha,\beta}[u]\\[10pt]
		&\,=\,\left\{(1+t)\frac{d}{dt}+n+\alpha+\beta+1\right\}D_1[u]+L^{\alpha,\beta}[u]+n(n+\alpha+\beta+1)\,u
	\end{align*}
	or, equivalently,
	$$
	L^{\alpha+1,\beta}D_1[u]\,=\,D_1\left(L^{\alpha,\beta}[u]+(n+\alpha+\beta+1)\,u\right)+L^{\alpha,\beta}[u]+n(n+\alpha+\beta+1)\,u.
	$$
	From \eqref{diffeq}, it follows that
	$$
	L^{\alpha+1,\beta}D_1[P_n^{(\alpha,\beta)}(t)]\,=\,-(n-1)\,(n+\alpha+\beta+1)\,D_1[P_n^{(\alpha,\beta)}(t)],
	$$
	which implies that $D_1[P_n^{(\alpha,\beta)}(t)]=c_n^{\alpha,\beta}\,P_{n-1}^{(\alpha+1,\beta)}(t)$. Writing,
	\begin{equation}\label{J7}
		D_1[u]\,=\,(t-1)\frac{du}{dt}-n\,u+2\frac{du}{dt},
	\end{equation}
	we can compare leading coefficients and obtain $c_n^{\alpha,\beta}\,=\,n+\beta$.
	
	Moreover, notice that we can write
	$$
	D_2[u]\,=\,-(t-1)^2\frac{du}{dt}-(n+\alpha+\beta+1)\,(t-1)\,u-2(t-1)\frac{du}{dt}-2\,\alpha\,u.
	$$
	Hence, 
	$$
	D_2\left[\left(\frac{t-1}{2} \right)^n\right]\,=\,-2(2n+\alpha+\beta+1)\,\left(\frac{t-1}{2} \right)^{n+1}-2(n+\alpha)\,\left(\frac{t-1}{2} \right)^{n},
	$$
	which implies that $D_2[P_n^{(\alpha,\beta)}(t)]$ is a polynomial of degree $n+1$. By the orthogonality of $\{P_{n}^{(\alpha-1,\beta)}(t)\}_{n\geqslant 0}$ and comparing leading coefficients, we have
	$$
	D_2[P_n^{(\alpha,\beta)}(t)]\,=\,-2(n+1)P_{n+1}^{(\alpha-1,\beta)}(t)+\sum_{k=0}^{n}d_k^{\alpha,\beta}\,P_{k}^{(\alpha-1,\beta)}(t),
	$$
	where
	$$
	d_k^{\alpha,\beta}\,h_k^{\alpha-1,\beta}\,=\,\int_{-1}^1D_2[P_n^{(\alpha,\beta)}(t)]\,P_{k}^{(\alpha-1,\beta)}(t)\,w_{\alpha-1,\beta}(t)\,dt, \quad 0\leqslant k \leqslant n.
	$$
	Integrating by parts as in the proof of Proposition \ref{univariateladders2}, we obtain
	$$
	d_k^{\alpha,\beta}\,h_k^{\alpha-1,\beta}\,=\,-\int_{-1}^1P_n^{(\alpha,\beta)}(t)\,D_1[P_{k}^{(\alpha-1,\beta)}(t)]\,w_{\alpha,\beta}(t)\,dt.
	$$
	On one hand, $D_1[P_{k}^{(\alpha-1,\beta)}(t)]$ is a polynomial of degree at most $k$ for $0\leqslant k \leqslant n-1$. From the orthogonality of the Jacobi polynomials, we deduce that $d_k^{\alpha,\beta}\,=\,0$ for $0\leqslant k \leqslant n-1$. On the other hand, by \eqref{J7}, we have
	$$
	D_1\left[\left(\frac{t-1}{2} \right)^n\right]\,=\,n\left(\frac{t-1}{2} \right)^{n-1}.
	$$
	Therefore, $D_1[P_{n}^{(\alpha-1,\beta)}(t)]$ is a polynomial of degree at most $n-1$ and, by orthogonality, $d_n^{\alpha,\beta}\,=\,0$. This proves \eqref{ladder4}.
\end{proof}

Using \eqref{ladder4} and \eqref{reflection}, we can show that the operators defined for $n\geqslant 0$ as
$$
E_1[u]\,=\,(1-t)\frac{du}{dt}+n\,u \quad \text{and} \quad E_2[u]\,=\,(1-t^2)\frac{du}{dt}+[(n+\alpha+1)(1+t)-\beta(1-t)]\,u,
$$
shift the parameter $\beta$ as well as the degree of the Jacobi polynomials.

\begin{prop}
	For $\alpha>-1$, the univariate Jacobi polynomials satisfy
	\begin{equation}\label{ladder5}
		\begin{aligned}
			E_1[P_n^{(\alpha,\beta)}(t)]&\,=\,(n+\alpha)\,P_{n-1}^{(\alpha,\beta+1)}(t), & \beta>-1,\\[10pt]
			E_2[P_n^{(\alpha,\beta)}(t)]&\,=\,-2(n+1)\,P_{n+1}^{(\alpha,\beta-1)}(t), & \beta>0.
		\end{aligned}
	\end{equation}
\end{prop}

We will now finish our list of ladder operators for Jacobi polynomials. We define the two operators with fixed coefficients:
$$
F_1[u]\,=\,(1+t)\frac{du}{dt}+\beta\,u \quad \text{and} \quad F_2[u]\,=\,(1-t)\frac{du}{dt}-\alpha\,u.
$$

\begin{prop}
	The univariate Jacobi polynomials satisfy
	\begin{equation}\label{ladder6}
		\begin{aligned}
			F_1[P_n^{(\alpha,\beta)}(t)]&\,=\,(n+\beta)\,P_{n}^{(\alpha+1,\beta-1)}(t), & \alpha,\,\beta-1>-1,\\[10pt]
			F_2[P_n^{(\alpha,\beta)}(t)]&\,=\,-(n+\alpha)\,P_{n}^{(\alpha-1,\beta+1)}(t), & \alpha-1,\beta>-1.
		\end{aligned}
	\end{equation}
\end{prop}

\begin{proof}
	We compute
	\begin{align*}
		L^{\alpha,\beta}F_1[u]-F_1L^{\alpha,\beta}[u]\,=\,2\frac{d}{dt}F_1[u],
	\end{align*}
or, equivalently,
$$
L^{\alpha+1,\beta-1}F_1[u]\,=\,F_1L^{\alpha,\beta}[u].
$$
From here together with \eqref{invariance}, we deduce \eqref{ladder6}. 
\end{proof}

\section{Ladder operators for Zernike polynomials}\label{ladderZernike}

In this section, we derive ladder operators for the generalized Zernike polynomials by directly employing ladder operators satisfied by the univariate Jacobi polynomials. We give several ladder operators that shift the parameter and degree of the complex generalized Zernike polynomials by at most one unit, that is, each ladder operator maps $Q^{\mu}_{k,j}(z,\bar{z})$ to $Q^{\eta}_{m,\ell}(z,\bar{z})$ with $|\mu-\eta|\leqslant 1$, $|k-m|\leqslant 1$, $|j-\ell|\leqslant 1$. 

The first two ladder operators for the complex generalized Zernike polynomials are deduced in the following theorem.

\begin{theorem}
	The complex generalized Zernike polynomials satisfy
	\begin{equation}\label{ladderZ1}
		\begin{aligned}
		&\frac{\partial}{\partial z}Q^{\mu}_{k,j}(z,\bar{z})\,=\,\frac{k\,(j+\mu+1)}{\mu+1}Q^{\mu+1}_{k-1,j}(z,\bar{z}), & \mu> -1,\\[10pt]
		&\left\{(1-z\bar{z})\frac{\partial}{\partial \bar{z}}-\mu\,z \right\}Q^{\mu}_{k,j}(z,\bar{z})\,=\,-\mu\,Q^{\mu-1}_{k+1,j}(z,\bar{z}), & \mu> 0.
		\end{aligned}
	\end{equation}
\end{theorem}
\begin{proof}
	Using \eqref{conjugation}, we compute
	\begin{align*}
	\frac{\partial}{\partial z}Q^{\mu}_{k,j}(z,\bar{z})\,=\,2\,\frac{k!}{(\mu+1)_k}\bar{z}^{j-k+1}\left(P_k^{(\mu,j-k)}\right)'(2z\bar{z}-1).
	\end{align*}
Notice that if we set $t=2r-1$ where $r\,=\,z\bar{z}$, then 
$$
A_1[P_k^{(\mu,j-k)}(t)]\,=\,2\left(P_k^{(\mu,j-k)}\right)'(2r-1).
$$
Then, from \eqref{ladder1}, we get
\begin{align*}
	\frac{\partial}{\partial z}Q^{\mu}_{k,j}(z,\bar{z})&\,=\,(j+\mu+1)\frac{k!}{(\mu+1)_k}\bar{z}^{j-k+1}\,P_{k-1}^{(\mu+1,j-k+1)}(2z\bar{z}-1)\\[10pt]
	&\,=\,\frac{k\,(j+\mu+1)}{\mu+1}Q^{\mu+1}_{k-1,j}(z,\bar{z}),
\end{align*} 
and we obtain the first identity in \eqref{ladderZ1}.

Using \eqref{conjugation} again, we compute
\begin{equation*}
	\begin{aligned}
		&\left\{(1-z\bar{z})\frac{\partial}{\partial \bar{z}}-\mu\,z \right\}Q^{\mu}_{k,j}(z,\bar{z})\\[10pt]
		&=\,\frac{k!\,\bar{z}^{j-k-1}}{(\mu+1)_k}\left[2(1-z\bar{z})z\bar{z}\left(P_k^{(\mu,j-k)}\right)'+\left[(j-k)(1-z\bar{z})-\mu\,z\bar{z}\right]P_k^{(\mu,j-k)}\right](2z\bar{z}-1).	
	\end{aligned}
\end{equation*}
The expression in brackets is precisely $\frac{1}{2}A_2[P_k^{(\mu,j-k)}(t)]$ under the same change of variable as before. Then, from \eqref{ladder1} it follows that
\begin{equation*}
	\begin{aligned}
		\left\{(1-z\bar{z})\frac{\partial}{\partial z}-\mu\,z \right\}Q^{\mu}_{k,j}(z,\bar{z})&\,=\,\frac{1}{2}\frac{k!}{(\mu+1)_k}\,\bar{z}^{j-k-1}A_2[P_k^{(\mu,j-k)}](2z\bar{z}-1)\\[10pt]
		&\,=\,-\frac{\,(k+1)!}{(\mu+1)_k}\,\bar{z}^{j-k-1}P_{k+1}^{(\mu-1,j-k-1)}(2z\bar{z}-1)\\[10pt]
		&\,=\,-\mu\,Q^{\mu-1}_{k+1,j}(z,\bar{z}),
	\end{aligned}
\end{equation*}
and the second identity in \eqref{ladderZ1} follows. 

 We remark that this result can also be obtained by computing with \eqref{zernikecomplex} and using $F_1$ and $F_2$ instead of $A_1$ and $A_2$, respectively. We omit the details.
\end{proof}

The following corollary is a consequence of the previous theorem and \eqref{invariance}.
\begin{corollary}
	The complex generalized Zernike polynomials satisfy
	\begin{equation}\label{ladderZ2}
		\begin{aligned}
			&\frac{\partial}{\partial \bar{z}}Q^{\mu}_{k,j}(z,\bar{z})\,=\,\frac{j\,(k+\mu+1)}{\mu+1}Q^{\mu+1}_{k,j-1}(z,\bar{z}), & \mu> -1,\\[10pt]
			&\left\{(1-z\bar{z})\frac{\partial}{\partial \bar{z}}-\mu\,\bar{z} \right\}Q^{\mu}_{k,j}(z,\bar{z})\,=\,-\mu\,Q^{\mu-1}_{k,j+1}(z,\bar{z}), & \mu> 0.
		\end{aligned}
	\end{equation}
\end{corollary}

The ladder operators in the following theorem shift the parameter $\mu$ by one unit, but keep the other two paramenter $k$ and $j$ unchanged and, hence, the total degree of the polynomials is preserved under the action of these operators. 

\begin{theorem}
	The complex generalized Zernike polynomials satisfy
	\begin{equation}\label{ladderZ3}
		\begin{aligned}
			&\left\{z\frac{\partial}{\partial z}+j+\mu+1 \right\}Q^{\mu}_{k,j}(z,\bar{z})\,=\,\frac{(k+\mu+1)\,(j+\mu+1)}{\mu+1}Q^{\mu+1}_{k,j}(z,\bar{z}), & \mu> -1,\\[10pt]
			&\left\{(1-z\bar{z})\,z\frac{\partial}{\partial z}-k\,(1-z\bar{z})-\mu \right\}Q^{\mu}_{k,j}(z,\bar{z})\,=\,-\mu\,Q^{\mu-1}_{k,j}(z,\bar{z}), & \mu> 0.
		\end{aligned}
	\end{equation}
\end{theorem}
\begin{proof}
	Using \eqref{zernikecomplex}, we compute
	\begin{align*}
		&\left\{z\frac{\partial}{\partial z}+j+\mu+1 \right\}Q^{\mu}_{k,j}(z,\bar{z})\\[10pt]
		&\,=\,\frac{j!}{(\mu+1)_j}z^{k-j}\left[2z\bar{z}\left(P_j^{(\mu,k-j)} \right)'+(k+\mu+1)P_j^{(\mu,k-j)} \right](2z\bar{z}-1),
	\end{align*}
where the expression in brackets is $B_1[P_j^{(\mu,k-j)}(t)]$ under the change of variables $t=2r-1$ where $r\,=\,z\bar{z}$. By \eqref{ladder2}, we have
	\begin{align*}
	\left\{z\frac{\partial}{\partial z}+j+\mu+1 \right\}Q^{\mu}_{k,j}(z,\bar{z})&\,=\,\frac{j!}{(\mu+1)_j}(k+\mu+1)z^{k-j}\,P_j^{(\mu+1,k-j)}(2z\bar{z}-1),\\[10pt]
	&\,=\,\frac{(k+\mu+1)\,(j+\mu+1)}{\mu+1}Q^{\mu+1}_{k,j}(z,\bar{z}).
\end{align*}

Now, we compute
\begin{align*}
	&\left\{(1-z\bar{z})\,z\frac{\partial}{\partial z}-k\,(1-z\bar{z})-\mu \right\}Q^{\mu}_{k,j}(z,\bar{z})\\[10pt]
	&\,=\,\frac{j!}{(\mu+1)_j}z^{k-j}\left[2(1-z\bar{z})z\bar{z}\left(P_j^{(\mu,k-j)} \right)'-\bigg(j\,(1-z\bar{z})+\alpha\bigg)\,P_j^{(\mu,k-j)} \right](2z\bar{z}-1).
\end{align*}
The expression in brackets is $\frac{1}{2}B_2[P_j^{(\mu,k-j)}(t)]$ under the same change of variable as before. Hence, from \eqref{ladder2}, we obtain
\begin{align*}
	&\left\{(1-z\bar{z})\,z\frac{\partial}{\partial z}-k\,(1-z\bar{z})-\mu \right\}Q^{\mu}_{k,j}(z,\bar{z})\\[10pt]
	&\,=\,-(j+\alpha)\frac{j!}{(\mu+1)_j}z^{k-j}\,P_j^{(\mu-1,k-j)}(2z\bar{z}-1)\,=\,-\mu\,Q^{\mu-1}_{k,j}(z,\bar{z}).
\end{align*}
We have now proved both equations in \eqref{ladderZ3}.
\end{proof}

\begin{corollary}
	The complex generalized Zernike polynomials satisfy
	\begin{equation}\label{ladderZ4}
		\begin{aligned}
			&\left\{\bar{z}\frac{\partial}{\partial \bar{z}}+k+\mu+1 \right\}Q^{\mu}_{k,j}(z,\bar{z})\,=\,\frac{(k+\mu+1)\,(j+\mu+1)}{\mu+1}Q^{\mu+1}_{k,j}(z,\bar{z}), & \mu> -1,\\[10pt]
			&\left\{(1-z\bar{z})\,\bar{z}\frac{\partial}{\partial \bar{z}}-j\,(1-z\bar{z})-\mu \right\}Q^{\mu}_{k,j}(z,\bar{z})\,=\,-\mu\,Q^{\mu-1}_{k,j}(z,\bar{z}), & \mu> 0.
		\end{aligned}
	\end{equation}
\end{corollary}

In contrast to the previous ladder operators, the follwing ones only shift the parameters $k$ and $j$, and leave $\mu$ unchanged.

\begin{theorem}
	For $\mu>-1$, the complex generalized Zernike polynomials satisfy
	\begin{equation}\label{ladderZ5}
		\begin{aligned}
			&\left\{(1-z\bar{z})\,\frac{\partial}{\partial z}+k\bar{z} \right\}Q^{\mu}_{k,j}(z,\bar{z})\,=\,k\,Q^{\mu}_{k-1,j}(z,\bar{z}),\\[10pt]
			&\left\{(1-z\bar{z})\frac{\partial}{\partial \bar{z}}-(k+\mu+1)z \right\}Q^{\mu}_{k,j}(z,\bar{z})\,=\,-(k+\mu+1)\,Q^{\mu}_{k+1,j}(z,\bar{z}).
		\end{aligned}
	\end{equation}
\end{theorem}

\begin{proof}
From the representation \eqref{zernikecomplex}, we get
\begin{align*}
	&\left\{(1-z\bar{z})\frac{\partial}{\partial \bar{z}}-(k+\mu+1)z \right\}Q^{\mu}_{k,j}(z,\bar{z})\\[10pt]
	&\,=\,\frac{j!}{(\mu+1)_j}z^{k-j+1}\left[2(1-z\bar{z})\left(P_j^{(\mu,k-j)}\right)'-(k+\mu+1)P_j^{(\mu,k-j)} \right](2z\bar{z}-1).
\end{align*}	
Since the expression in brackets coincides with $C_1[P_j^{(\mu,k-j)}(t)]$ where $t=2r-1$ and $r\,=\,z\bar{z}$, it follows from \eqref{ladder3} that
\begin{align*}
	&\left\{(1-z\bar{z})\frac{\partial}{\partial \bar{z}}-(k+\mu+1)z \right\}Q^{\mu}_{k,j}(z,\bar{z})\\[10pt]
	&\,=\,-(k+\mu+1)\frac{j!}{(\mu+1)_j}z^{k-j+1}P_j^{(\mu,k-j+1)}(2z\bar{z}-1)\,=\,-(k+\mu+1)\,Q^{\mu}_{k+1,j}(z,\bar{z}).
\end{align*}

Similarly, we compute
\begin{align*}
	&\left\{(1-z\bar{z})\,\frac{\partial}{\partial z}+k\bar{z} \right\}Q^{\mu}_{k,j}(z,\bar{z})\\[10pt]
	&\,=\,\frac{j!}{(\mu+1)_j}z^{k-j-1}\left[2(1-z\bar{z})z\bar{z}\left(P_j^{(\mu,k-j)}\right)'+(jz\bar{z}+k-j)P_j^{(\mu,k-j)} \right](2z\bar{z}-1).
\end{align*}
Since the expression in brackets coincides with $\frac{1}{2}C_2[P_j^{(\mu,k-j)}(t)]$ under the same change of variables as before, it follows from \eqref{ladder3} that
\begin{align*}
	\left\{(1-z\bar{z})\,\frac{\partial}{\partial z}+k\bar{z} \right\}Q^{\mu}_{k,j}(z,\bar{z})&\,=\,\frac{k\,j!}{(\mu+1)_j}z^{k-j-1}P_j^{(\mu,k-j-1)}(2z\bar{z}-1)\\[10pt]
	&\,=\,k\,Q^{\mu}_{k-1,j}(z,\bar{z}).
\end{align*}
This proves \eqref{ladderZ5}. This result can also be obtained by computing with \eqref{conjugation} and using $E_2$ and $E_1$ instead of $C_1$ and $C_2$, respectively. We omit the details.
\end{proof}

\begin{corollary}
	For $\mu>-1$, the complex generalized Zernike polynomials satisfy
	\begin{equation}\label{ladderZ6}
		\begin{aligned}
		&\left\{(1-z\bar{z})\,\frac{\partial}{\partial \bar{z}}+jz \right\}Q^{\mu}_{k,j}(z,\bar{z})\,=\,j\,Q^{\mu}_{k,j-1}(z,\bar{z}),\\[10pt]
		&\left\{(1-z\bar{z})\frac{\partial}{\partial z}-(j+\mu+1)\bar{z} \right\}Q^{\mu}_{k,j}(z,\bar{z})\,=\,-(j+\mu+1)\,Q^{\mu}_{k,j+1}(z,\bar{z}).
		\end{aligned}
	\end{equation}
\end{corollary}

The following ladder operators shift all three parameters by one unit. In particular, both parameters $k$ and $j$ are shifted by the same amount and, consequently, the total degree of the polynomials is shifted by two units.

\begin{theorem}
	The complex generalized Zernike polynomials satisfy
	\begin{equation}\label{ladderZ7}
		\begin{aligned}
			&\left\{z\,\frac{\partial}{\partial z}-k \right\}Q^{\mu}_{k,j}(z,\bar{z})\,=\,\frac{k\,j}{\mu+1}\,Q^{\mu+1}_{k-1,j-1}(z,\bar{z}),\\[10pt]
			&\left\{(1-z\bar{z})z\frac{\partial}{\partial z}+(j+1)(1-z\bar{z})-\mu z\bar{z} \right\}Q^{\mu}_{k,j}(z,\bar{z})\,=\,-\mu\,Q^{\mu-1}_{k+1,j+1}(z,\bar{z}).
		\end{aligned}
	\end{equation}
\end{theorem}
\begin{proof}
	Observe that under the change of variable $t=2z\bar{z}-1$, we have
	\begin{align*}
		&D_1[P_{k}^{(\mu,j-k)}(t)]\,=\,\left\{2z\bar{z}\left(P_{k}^{(\mu,j-k)} \right)'-k\,P_{k}^{(\mu,j-k)}\right\}(2z\bar{z}-1),\\[10pt]
		&D_2[P_{j}^{(\mu,k-j)}(t)]\,=\,\left\{4(1-z\bar{z})z\bar{z}\left( P_{j}^{(\mu,k-j)}\right)'+2[(k+1)(1-z\bar{z})-\mu\,z\bar{z}]P_{j}^{(\mu,k-j)} \right\}(2z\bar{z}-1).
	\end{align*}
From here, we get
\begin{align*}
	\left\{z\,\frac{\partial}{\partial z}-k \right\}Q^{\mu}_{k,j}(z,\bar{z})&\,=\,\frac{k!}{(\mu+1)_k}\bar{z}^{j-k}D_1[P_{k}^{(\mu,j-k)}](2z\bar{z}-1)\\[10pt]
	&\,=\,j\frac{k!}{(\mu+1)_k}\bar{z}^{j-k}P_{k-1}^{(\mu+1,j-k)}(2z\bar{z}-1)\\[10pt]
	&\,=\,\frac{k\,j}{\mu+1}\,Q^{\mu+1}_{k-1,j-1}(z,\bar{z}),
\end{align*}
and
\begin{align*}
	&\left\{(1-z\bar{z})z\frac{\partial}{\partial z}+(j+1)(1-z\bar{z})-\mu z\bar{z} \right\}Q^{\mu}_{k,j}(z,\bar{z})\\[10pt]
	&\qquad \qquad \qquad \qquad \,=\,\frac{1}{2}\frac{j!}{(\mu+1)_j}z^{k-j}D_2[P_{j}^{(\mu,k-j)}](2z\bar{z}-1)\\[10pt]
	&\qquad \qquad \qquad \qquad  \,=\,-(j+1)\frac{j!}{(\mu+1)_j}z^{k-j}P_{j+1}^{(\mu-1,k-j)}(2z\bar{z}-1)\\[10pt]
	&\qquad \qquad \qquad \qquad \,=\,-\mu\,Q^{\mu-1}_{k+1,j+1}(z,\bar{z}),
\end{align*}
which proves \eqref{ladderZ7}. 
\end{proof}

\begin{corollary}
	For $\mu>-1$, the complex generalized Zernike polynomials satisfy
		\begin{equation}\label{ladderZ8}
		\begin{aligned}
			&\left\{\bar{z}\,\frac{\partial}{\partial \bar{z}}-j \right\}Q^{\mu}_{k,j}(z,\bar{z})\,=\,\frac{k\,j}{\mu+1}\,Q^{\mu+1}_{k-1,j-1}(z,\bar{z}),\\[10pt]
			&\left\{(1-z\bar{z})\bar{z}\frac{\partial}{\partial \bar{z}}+(k+1)(1-z\bar{z})-\mu z\bar{z} \right\}Q^{\mu}_{k,j}(z,\bar{z})\,=\,-\mu\,Q^{\mu-1}_{k+1,j+1}(z,\bar{z}).
		\end{aligned}
	\end{equation}
\end{corollary}

\begin{center}
	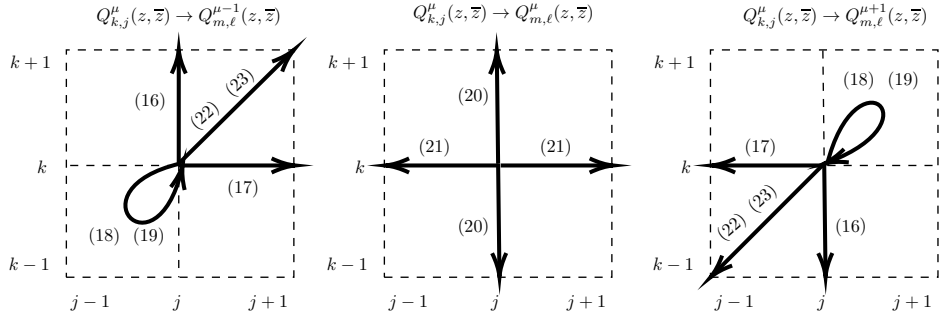
\begin{figure}[htb]
		\centering
		\resizebox{\textwidth}{!}{

			\tikzset{every picture/.style={line width=0.75pt}} 
			
			\begin{tikzpicture}[x=0.75pt,y=0.75pt,yscale=-1,xscale=1]
				
				\draw  [dash pattern={on 4.5pt off 4.5pt}] (49.17,70) -- (218,70) -- (218,238.83) -- (49.17,238.83) -- cycle ;
				\draw  [dash pattern={on 4.5pt off 4.5pt}]  (132.58,70) -- (132.58,238.83) ;
				\draw  [dash pattern={on 4.5pt off 4.5pt}]  (52,155.83) -- (219,155.83) ;
				\draw [line width=2.25]    (132.58,155.83) -- (132.58,70) ;
				\draw [shift={(132.58,70)}, rotate = 90] [color={rgb, 255:red, 0; green, 0; blue, 0 }  ][line width=2.25]    (17.49,-5.26) .. controls (11.12,-2.23) and (5.29,-0.48) .. (0,0) .. controls (5.29,0.48) and (11.12,2.23) .. (17.49,5.26)   ;
				\draw [line width=2.25]    (132.58,155.83) -- (218,155.83) ;
				\draw [shift={(219,155.83)}, rotate = 180] [color={rgb, 255:red, 0; green, 0; blue, 0 }  ][line width=2.25]    (17.49,-5.26) .. controls (11.12,-2.23) and (5.29,-0.48) .. (0,0) .. controls (5.29,0.48) and (11.12,2.23) .. (17.49,5.26)   ;
				\draw [line width=2.25]    (132.58,154.42) -- (218,70) ;
				\draw [shift={(218,70)}, rotate = 135] [color={rgb, 255:red, 0; green, 0; blue, 0 }  ][line width=2.25]    (17.49,-5.26) .. controls (11.12,-2.23) and (5.29,-0.48) .. (0,0) .. controls (5.29,0.48) and (11.12,2.23) .. (17.49,5.26)   ;
				\draw [line width=2.25]    (133.58,154.42) .. controls (52.24,173.54) and (115.05,239.8) .. (134.64,159.6) ;
				\draw [shift={(135.5,155.83)}, rotate = 102.14] [color={rgb, 255:red, 0; green, 0; blue, 0 }  ][line width=2.25]    (17.49,-5.26) .. controls (11.12,-2.23) and (5.29,-0.48) .. (0,0) .. controls (5.29,0.48) and (11.12,2.23) .. (17.49,5.26)   ;
				\draw  [dash pattern={on 4.5pt off 4.5pt}] (285.17,70) -- (454,70) -- (454,238.83) -- (285.17,238.83) -- cycle ;
				\draw  [dash pattern={on 4.5pt off 4.5pt}]  (368.58,70.42) -- (370.58,238.42) ;
				\draw  [dash pattern={on 4.5pt off 4.5pt}]  (288,155.83) -- (455,155.83) ;
				\draw [line width=2.25]    (369.58,154.42) -- (370.54,234.42) ;
				\draw [shift={(370.58,238.42)}, rotate = 269.32] [color={rgb, 255:red, 0; green, 0; blue, 0 }  ][line width=2.25]    (17.49,-5.26) .. controls (11.12,-2.23) and (5.29,-0.48) .. (0,0) .. controls (5.29,0.48) and (11.12,2.23) .. (17.49,5.26)   ;
				\draw [line width=2.25]    (369.58,154.42) -- (368.63,74.42) ;
				\draw [shift={(368.58,70.42)}, rotate = 89.32] [color={rgb, 255:red, 0; green, 0; blue, 0 }  ][line width=2.25]    (17.49,-5.26) .. controls (11.12,-2.23) and (5.29,-0.48) .. (0,0) .. controls (5.29,0.48) and (11.12,2.23) .. (17.49,5.26)   ;
				\draw [line width=2.25]    (369.5,155.83) -- (290,155.83) ;
				\draw [shift={(286,155.83)}, rotate = 360] [color={rgb, 255:red, 0; green, 0; blue, 0 }  ][line width=2.25]    (17.49,-5.26) .. controls (11.12,-2.23) and (5.29,-0.48) .. (0,0) .. controls (5.29,0.48) and (11.12,2.23) .. (17.49,5.26)   ;
				\draw [line width=2.25]    (371.5,155.83) -- (451,155.83) ;
				\draw [shift={(455,155.83)}, rotate = 180] [color={rgb, 255:red, 0; green, 0; blue, 0 }  ][line width=2.25]    (17.49,-5.26) .. controls (11.12,-2.23) and (5.29,-0.48) .. (0,0) .. controls (5.29,0.48) and (11.12,2.23) .. (17.49,5.26)   ;
				\draw  [dash pattern={on 4.5pt off 4.5pt}] (527.17,70) -- (696,70) -- (696,238.83) -- (527.17,238.83) -- cycle ;
				\draw  [dash pattern={on 4.5pt off 4.5pt}]  (610.58,70.42) -- (612.58,238.42) ;
				\draw  [dash pattern={on 4.5pt off 4.5pt}]  (530,155.83) -- (697,155.83) ;
				\draw [line width=2.25]    (611.58,154.42) -- (612.54,234.42) ;
				\draw [shift={(612.58,238.42)}, rotate = 269.32] [color={rgb, 255:red, 0; green, 0; blue, 0 }  ][line width=2.25]    (17.49,-5.26) .. controls (11.12,-2.23) and (5.29,-0.48) .. (0,0) .. controls (5.29,0.48) and (11.12,2.23) .. (17.49,5.26)   ;
				\draw [line width=2.25]    (611.58,154.42) -- (530,236) ;
				\draw [shift={(527.17,238.83)}, rotate = 315] [color={rgb, 255:red, 0; green, 0; blue, 0 }  ][line width=2.25]    (17.49,-5.26) .. controls (11.12,-2.23) and (5.29,-0.48) .. (0,0) .. controls (5.29,0.48) and (11.12,2.23) .. (17.49,5.26)   ;
				\draw [line width=2.25]    (611.5,155.83) -- (532,155.83) ;
				\draw [shift={(528,155.83)}, rotate = 360] [color={rgb, 255:red, 0; green, 0; blue, 0 }  ][line width=2.25]    (17.49,-5.26) .. controls (11.12,-2.23) and (5.29,-0.48) .. (0,0) .. controls (5.29,0.48) and (11.12,2.23) .. (17.49,5.26)   ;
				\draw [line width=2.25]    (613.5,155.83) .. controls (634.68,70.47) and (695.15,122.55) .. (615.33,153.04) ;
				\draw [shift={(611.58,154.42)}, rotate = 340.5] [color={rgb, 255:red, 0; green, 0; blue, 0 }  ][line width=2.25]    (17.49,-5.26) .. controls (11.12,-2.23) and (5.29,-0.48) .. (0,0) .. controls (5.29,0.48) and (11.12,2.23) .. (17.49,5.26)   ;
				
				\draw (70,34.4) node [anchor=north west][inner sep=0.75pt]    {$Q_{k,j}^{\mu }( z,\overline{z})\rightarrow Q_{m,\ell }^{\mu -1}( z,\overline{z})$};
				\draw (26,150.4) node [anchor=north west][inner sep=0.75pt]    {$k$};
				\draw (6,72.4) node [anchor=north west][inner sep=0.75pt]    {$k+1$};
				\draw (5,222.4) node [anchor=north west][inner sep=0.75pt]    {$k-1$};
				\draw (51,250.4) node [anchor=north west][inner sep=0.75pt]    {$j-1$};
				\draw (127,250.4) node [anchor=north west][inner sep=0.75pt]    {$j$};
				\draw (183,250.4) node [anchor=north west][inner sep=0.75pt]    {$j+1$};
				\draw (98,100.4) node [anchor=north west][inner sep=0.75pt]    {\eqref{ladderZ1}};
				\draw (167,164.4) node [anchor=north west][inner sep=0.75pt]    {\eqref{ladderZ2}};
				\draw (63.79,200.83) node [anchor=north west][inner sep=0.75pt]  [rotate=-358.36]  {\eqref{ladderZ3}};
				\draw (97.04,200.33) node [anchor=north west][inner sep=0.75pt]  [rotate=-0.27]  {\eqref{ladderZ4}};
				\draw (136.76,121.47) node [anchor=north west][inner sep=0.75pt]  [rotate=-319.17]  {\eqref{ladderZ7}};
				\draw (162.78,97.45) node [anchor=north west][inner sep=0.75pt]  [rotate=-317.79]  {\eqref{ladderZ8}};
				\draw (305,34.4) node [anchor=north west][inner sep=0.75pt]    {$Q_{k,j}^{\mu }( z,\overline{z})\rightarrow Q_{m,\ell }^{\mu }( z,\overline{z})$};
				\draw (262,150.4) node [anchor=north west][inner sep=0.75pt]    {$k$};
				\draw (242,72.4) node [anchor=north west][inner sep=0.75pt]    {$k+1$};
				\draw (241,222.4) node [anchor=north west][inner sep=0.75pt]    {$k-1$};
				\draw (287,250.4) node [anchor=north west][inner sep=0.75pt]    {$j-1$};
				\draw (363,250.4) node [anchor=north west][inner sep=0.75pt]    {$j$};
				\draw (419,250.4) node [anchor=north west][inner sep=0.75pt]    {$j+1$};
				\draw (338,96.4) node [anchor=north west][inner sep=0.75pt]    {\eqref{ladderZ5}};
				\draw (399,134.4) node [anchor=north west][inner sep=0.75pt]    {\eqref{ladderZ6}};
				\draw (337.79,193.83) node [anchor=north west][inner sep=0.75pt]  [rotate=-358.36]  {\eqref{ladderZ5}};
				\draw (309.04,134.33) node [anchor=north west][inner sep=0.75pt]  [rotate=-0.27]  {\eqref{ladderZ6}};
				\draw (550,34.4) node [anchor=north west][inner sep=0.75pt]    {$Q_{k,j}^{\mu }( z,\overline{z})\rightarrow Q_{m,\ell }^{\mu +1}( z,\overline{z})$};
				\draw (504,150.4) node [anchor=north west][inner sep=0.75pt]    {$k$};
				\draw (484,72.4) node [anchor=north west][inner sep=0.75pt]    {$k+1$};
				\draw (483,222.4) node [anchor=north west][inner sep=0.75pt]    {$k-1$};
				\draw (529,250.4) node [anchor=north west][inner sep=0.75pt]    {$j-1$};
				\draw (605,250.4) node [anchor=north west][inner sep=0.75pt]    {$j$};
				\draw (661,250.4) node [anchor=north west][inner sep=0.75pt]    {$j+1$};
				\draw (624,84.4) node [anchor=north west][inner sep=0.75pt]    {\eqref{ladderZ3}};
				\draw (617.79,193.83) node [anchor=north west][inner sep=0.75pt]  [rotate=-358.36]  {\eqref{ladderZ1}};
				\draw (551.04,134.33) node [anchor=north west][inner sep=0.75pt]  [rotate=-0.27]  {\eqref{ladderZ2}};
				\draw (657.04,84.33) node [anchor=north west][inner sep=0.75pt]  [rotate=-0.27]  {\eqref{ladderZ4}};
				\draw (527.76,204.47) node [anchor=north west][inner sep=0.75pt]  [rotate=-319.17]  {\eqref{ladderZ7}};
				\draw (550.78,183.45) node [anchor=north west][inner sep=0.75pt]  [rotate=-317.79]  {\eqref{ladderZ8}};		
				
			\end{tikzpicture}
		}
		\caption{Illustration of how the ladder operators increase or decrease the parameters in $Q_{k,j}^{\mu}(z,\bar{z})$.}
	\end{figure}
\end{center}

\section{Differential and recurrence relations for $Q^{\mu}_{k,j}(z,\bar{z})$}\label{recurrence}

We can combine the ladder operators introduced in the previous section to deduce differential and recurrence relations satisfied by the complex generalized Zernike polynomials. Some of the relations involve polynomials with different parameters.

First, we define the following differential operator
$$
L_{\mu}\,=\,2(1-z\bar{z})\frac{\partial^2}{\partial z\,\partial \bar{z}}-(\mu+1)\,\left(z\frac{\partial }{\partial z}+\bar{z}\frac{\partial}{\partial \bar{z}}\right).
$$

\begin{prop}\label{eigen}
	The complex generalized Zernike polynomials satisfy
	$$
	L_{\mu}Q^{\mu}_{k,j}(z,\bar{z})\,=\,\lambda_{k,j}^{\mu}\,Q^{\mu}_{k,j}(z,\bar{z}),
	$$
	where $\lambda_{k,j}^{\mu}=-2kj-(\mu+1)(k+j)$.
\end{prop}
\begin{proof}
	Observe that we can write
	$$
	L_{\mu}\,=\,\left\{(1-z\bar{z})\frac{\partial}{\partial \bar{z}}-(\mu+1)z \right\}\frac{\partial}{\partial z}+\left\{(1-z\bar{z})\frac{\partial}{\partial z}-(\mu+1)\bar{z} \right\}\frac{\partial}{\partial \bar{z}}.
	$$
	Then, the result follows from \eqref{ladderZ1} and \eqref{ladderZ2}.
\end{proof}

Notice that the differential equation $L_{\mu}P\,=\,\lambda^{\mu}\,P$ makes sense for $\mu\,=\,-1$. In fact, its polynomial solutions can be given in terms of the complex generalized Zernike polynomials (\cite{PX09}). However, these polynomials solutions are not standard orthogonal polynomials (\cite{GM23}).

\begin{prop}
	For $k,j\geqslant 0$, define
	\begin{equation}\label{sobolev1}
			Q^{-1}_{0,0}(z,\bar{z})\,=\,1,\qquad 
			Q^{-1}_{k,j}(z,\bar{z})\,=\,(1-z\bar{z})\,Q^{1}_{k-1,j-1}(z,\bar{z}), \quad k,j\geqslant 1.
	\end{equation}
Then $\{Q^{-1}_{k,j}(z,\bar{z}) \}_{k,j\geqslant 0}$ satisfy
$$
L_{-1}Q^{-1}_{k,j}(z,\bar{z})\,=\,\lambda^{-1}_{k,j}Q^{-1}_{k,j}(z,\bar{z}),
$$
with $\lambda_{k,j}^{-1}\,=\,-2kj$.
\end{prop}
\begin{proof}
	Note that
	$$
	L_{-1}\,=\,2(1-z\bar{z})\frac{\partial^2}{\partial z\,\partial \bar{z}}.
	$$
	Clearly, since $\lambda_{0,0}^{-1}\,=\,0$, we have $L_{-1}Q^{-1}_{0,0}(z,\bar{z})\,=\,\lambda^{-1}_{0,0}Q^{-1}_{0,0}(z,\bar{z})$. For $k,j\geqslant 1$, we have
	\begin{align*}
		L_{-1}Q^{-1}_{k,j}(z,\bar{z})&\,=\,2(1-z\bar{z})\frac{\partial^2}{\partial z\,\partial \bar{z}}(1-z\bar{z})Q^{1}_{k-1,j-1}(z,\bar{z})\\[10pt]
		&\,=\,(1-z\bar{z})\left\{2(1-z\bar{z})\frac{\partial^2}{\partial z\,\partial \bar{z}}-2\,\left(z\frac{\partial }{\partial z}+\bar{z}\frac{\partial}{\partial \bar{z}}\right)-2\right\}Q^{1}_{k-1,j-1}(z,\bar{z})\\[10pt]
		&\,=\,(1-z\bar{z})\left\{L_1-2\right\}Q^{1}_{k-1,j-1}(z,\bar{z}).
	\end{align*}
From Porposition \ref{eigen}, we get
\begin{align*}
	L_{-1}Q^{-1}_{k,j}(z,\bar{z})\,=\,(\lambda_{k-1,j-1}^{1}-2)Q^{-1}_{k,j}(z,\bar{z})\,=\,\lambda_{k,j}^{-1}Q^{-1}_{k,j}(z,\bar{z}),
\end{align*}
which proves the announced result.
\end{proof}

We remark that it was shown in \cite{PX09} that if $\mu\leqslant -2$ is an integer, then $L_{\mu}$ is not guaranteed to have a comples polynomial system solution (see also Remark 6.1 in \cite{GM23}).

The basic three term relations for multivariate orthogonal polynomials (\cite{DX14}) can be deduced for the complex generalized Zernike polynomials.

\begin{prop}
	The complex generalized Zernike polynomials satisfy the following three term relations
	\begin{align*}
		(k+j+\mu+1)z\,Q_{k,j}^{\mu}(z,\bar{z})&\,=\,(k+\mu+1)\,Q_{k+1,j}^{\mu}(z,\bar{z})+j\,Q_{k,j-1}^{\mu}(z,\bar{z}),\\[10pt]
		(k+j+\mu+1)\bar{z}\,Q_{k,j}^{\mu}(z,\bar{z})&\,=\,(j+\mu+1)\,Q_{k,j+1}^{\mu}(z,\bar{z})+k\,Q_{k-1,j}^{\mu}(z,\bar{z}).
\end{align*}
\end{prop}
\begin{proof}
	The result follows from writing
	\begin{align*}
		(k+j+\mu+1)z&\,=\,\left\{(1-z\bar{z})\frac{\partial}{\partial z}+j\,\bar{z}\right\}-\left\{(1-z\bar{z})\frac{\partial}{\partial z}-(k+\mu+1)\,\bar{z}\right\},\\[10pt]
		(k+j+\mu+1)\bar{z}&\,=\,\left\{(1-z\bar{z})\frac{\partial}{\partial \bar{z}}+k\,z\right\}-\left\{(1-z\bar{z})\frac{\partial}{\partial \bar{z}}-(j+\mu+1)\,z\right\},
	\end{align*}
and the using \eqref{ladderZ5} and \eqref{ladderZ6}.
\end{proof}

We can also deduce the linear relation between families of complex generalized Zernike polynomials with parameters $\mu$ and $\mu+1$. We remark that this linear relation is of fixed length for all values of $k$ and $j$.

\begin{prop}
	The complex generalized Zernike polynomials satisfy
	$$
	(k+j+\mu+1)\,Q_{k,j}^{\mu}(z,\bar{z})\,=\,\frac{(k+\mu+1)(j+\mu+1)}{\mu+1}Q_{k,j}^{\mu+1}(z,\bar{z})-\frac{k\,j}{\mu+1}Q_{k-1,j-1}^{\mu+1}(z,\bar{z}).
	$$
\end{prop}
\begin{proof}
	Since
	$$
	(k+j+\mu+1)\,Q_{k,j}^{\mu}(z,\bar{z})\,=\,\left\{z\frac{\partial}{\partial z}+j+\mu+1\right\}Q_{k,j}^{\mu}(z,\bar{z})-\left\{z\frac{\partial}{\partial z}-k \right\}Q_{k,j}^{\mu}(z,\bar{z}),
	$$
	the result follows from \eqref{ladderZ3} and \eqref{ladderZ7}.
\end{proof}

Ladder operators can be combined to obtain the so-called structure relations, that is, linear relations of fixed length involving the partial derivatives of the polynomials.

\begin{prop}
	The complex generalized Zernike polynomials satisfy the following structure relations
	$$
	(k+j+\mu+1)(1-z\bar{z})\frac{\partial}{\partial z}Q_{k,j}^{\mu}(z,\bar{z})\,=\,k(j+\mu+1)\left(Q_{k-1,j}^{\mu}(z,\bar{z})-Q_{k,j+1}^{\mu}(z,\bar{z}) \right),
	$$
	and
	$$
	(k+j+\mu+1)(1-z\bar{z})\frac{\partial}{\partial \bar{z}}Q_{k,j}^{\mu}(z,\bar{z})\,=\,j(k+\mu+1)\left(Q_{k,j-1}^{\mu}(z,\bar{z})-Q_{k+1,j}^{\mu}(z,\bar{z}) \right).
	$$
\end{prop}
\begin{proof}
	The first structure relation follows from writing
	$$
	(k+j+\mu+1)(1-z\bar{z})\frac{\partial}{\partial z}\,=\,k\left\{(1-z\bar{z})\frac{\partial}{\partial z}-(j+\mu+1)\bar{z}\right\}+(j+\mu+1)\left\{\frac{\partial}{\partial z}+k\bar{z}\right\},
	$$
	and using \eqref{ladderZ5} and \eqref{ladderZ6}. The second structure relation follows similarly.
\end{proof}

From the structure relations and \eqref{ladderZ1}, we get the following corollary.

\begin{corollary}
	The complex generalized Zernike polynomials satisfy
	$$
	(k+j+\mu+2)(1-z\bar{z})Q_{k,j}^{\mu+1}(z,\bar{z})\,=\,(\mu+1)\left(Q_{k,j}^{\mu}(z,\bar{z})-Q_{k+1,j+1}^{\mu}(z,\bar{z}) \right).
	$$
\end{corollary}

Since the right hand side of the first equation in \eqref{ladderZ7} and the first equation in \eqref{ladderZ8} are equal, we readily get the following fundamental relation.

\begin{prop}
	The complex generalized Zernike polynomials satisfy
	$$
	\left\{z\frac{\partial}{\partial z}-\bar{z}\frac{\partial}{\partial \bar{z}} \right\}Q_{k,j}^{\mu}(z,\bar{z})\,=\,(k-j)\,Q_{k,j}^{\mu}(z,\bar{z}).
	$$
\end{prop}

\section{A note about Sobolev orthogonality}
Univariate orthogonal polynomials with respect to a Sobolev inner product (i.e., an inner product involving the derivatives of the polynomials) have been studied extensively in the past few decades. In contrast to one variable, Sobolev orthogonal polynomials of several variables are studied only recently. We refer the interested reader to the survey \cite{MX15} for a recent presentation of the state of the art on Sobolev orthogonal polynomials. Several authors have taken an interest on Sobolev orthogonal polynomials on the unit ball of $\mathbb{R}^n$ which, clearly, include the generalized Zernike polynomials when $n=2$ (see, for instance, \cite{DPP13, DFLPP16,LX14, LPP21, MPPR23, MPR23, PPX13, Xu08, Xu06}). In this section, we illustrate the use of ladder operators in the study of Sobolev orthogonal polynomials on the disk.

\subsection{Example 1}
Consider the following Sobolev inner product
$$
(f,g)_1\,=\,\frac{\lambda}{\pi}\int_{\mathbf{D}}\frac{\partial f}{\partial z}(z,\bar{z})\,\overline{\frac{\partial g}{\partial z}(z,\bar{z})}\,dz+\frac{1}{\pi}\int_0^{2\pi}f(e^{i\theta},e^{-i\theta})\,\overline{g(e^{i\theta},e^{-i\theta})}\,d\theta, \quad \lambda>0.
$$
We remark that this inner product is a complex version of one of the Sobolev inner products studied in \cite{Xu08} given by
$$
(\tilde{f},\tilde{g})\,=\,\frac{\lambda}{\pi}\int_{\mathbf{D}}\nabla \tilde{f}(x,y) \cdot \nabla \tilde{g}(x,y)\,dxdy+\frac{1}{\pi} \int_0^{2\pi}\tilde{f}(\cos\theta,\sin\theta)\,\tilde{g}(\cos\theta,\sin \theta)\,d\theta, \quad \lambda>0,
$$
for real-valued functions $\tilde{f}(x,y)\,=\,f(z,\bar{z})$ and $\tilde{g}(x,y)\,=\,g(z,\bar{z})$ under the change of variable $z\,=\,x+iy\,=\,r\,e^{i\theta}$.

Now we use ladder operators for complex generalized Zernike polynomials to show that the polynomials defined in \eqref{sobolev1} are orthogonal with respect to $(\cdot,\cdot)_1$. 

\begin{lemma}\label{laddersobolev1}
	The polynomials defined in \eqref{sobolev1} satisfy
	$$
	\frac{\partial}{\partial z}Q_{k,j}^{-1}(z,\bar{z})\,=\,-Q_{k-1,j}^0(z,\bar{z}), \quad k,j\geqslant 1.
	$$
\end{lemma}

\begin{proof}
	We compute,
	$$
	\frac{\partial}{\partial z}Q_{k,j}^{-1}(z,\bar{z})\,=\,\left\{(1-z\bar{z})\frac{\partial}{\partial z}-\bar{z} \right\}Q_{k-1,j-1}^1(z,\bar{z}).
	$$
	Our result follows from \eqref{ladderZ2}.
\end{proof}

\begin{prop}
	The polynomials defined in \eqref{sobolev1} constitute a mutually orthogonal polynomial system with respect to $(\cdot,\cdot)_1$. Moreover,
	$$
	(Q_{k,j}^{-1},Q_{m,\ell}^{-1})_1\,=\,\tilde{h}_{k,j}(\lambda)\,\delta_{k,m}\,\delta_{j,\ell},
	$$
	with
	$$
	\tilde{h}_{k,j}(\lambda)\,=\,\left\{ \begin{array}{ll}
		2, & k=j=0,\\[10pt]
		\lambda\,h^0_{k-1,j}, & k,j\geqslant 1,
		\end{array}\right.
	$$
	where $h^0_{k-1,j}$ is given in \eqref{norm}. 
\end{prop}

\begin{proof}
	Clearly,
	$$
	(Q_{0,0}^{-1},Q_{0,0}^{-1})_1\,=\,(1,1)_1\,=\,\frac{1}{\pi}\int_0^{2\pi}d\theta\,=\,2.
	$$
	For $k,j\geqslant 1$, 
	$$
	Q_{k,j}^{-1}(e^{i\theta},e^{-i\theta})\,=\,0,
	$$
	since the factor $(1-z\bar{z})$ vanishes when $z=e^{i\theta}$, $0\leqslant \theta \leqslant 2\pi$. Therefore, 
	$$
	(Q_{k,j}^{-1},Q_{0,0}^{-1})_1\,=\,0, \quad k,j\geqslant 0.
	$$
	Using Lemma \eqref{laddersobolev1}, we get
	$$
	(Q_{k,j}^{-1},Q_{m,\ell}^{-1})_1\,=\,\frac{\lambda}{\pi}\int_{\mathbf{D}}Q_{k-1,j}^0(z,\bar{z})\,\overline{Q_{m-1,\ell}^0(z,\bar{z})}\,dz\,=\,\lambda\,h^{0}_{k-1,j}\,\delta_{k,m}\,\delta_{j,\ell},
	$$
	where the last equality follows from \eqref{orthogonality}.
\end{proof}

\subsection{Example 2}

Now consider the following Sobolev inner product
$$
(f,g)_2\,=\,\frac{1}{\pi}\int_{\mathbf{D}}\frac{\partial^2}{\partial z \partial \bar{z}}[(1-z\bar{z})f(z,\bar{z})]\,\overline{\frac{\partial^2 }{\partial z \partial \bar{z}}[(1-z\bar{z})g(z,\bar{z})] }dz.
$$
Recall that the two-dimensional Laplace operator can be represented in complex variables as
$$
\Delta\,=\,4\frac{\partial^2}{\partial z \partial \bar{z}}.
$$
Hence, for real-valued functions $\tilde{f}(x,y)\,=\,f(z,\bar{z})$ and $\tilde{g}(x,y)\,=\,g(z,\bar{z})$, this inner product is equivalent to 
$$
(\tilde{f},\tilde{g})_{\Delta}\,=\,\frac{1}{16\pi}\int_{\mathbf{D}}\Delta[(1-x^2-y^2)\tilde{f}(x,y)]\,\Delta[(1-x^2-y^2)\tilde{g}(x,y)]\,dxdy.
$$
A mutually orthogonal polynomial system with respect to $(\cdot,\cdot)_{\Delta}$ is constructed and studied in \cite{Xu06}. Here, we study a complex mutually orthogonal polynomial system relative to $(\cdot,\cdot)_2$ via ladder operators.

Define the polynomials
	\begin{align}\label{sobolev2}
		U_{0,0}(z,\bar{z})\,=\,1, \quad 		U_{k,j}(z,\bar{z})\,=\,(1-z\bar{z})Q^2_{k-1,j-1}(z,\bar{z}), \quad k,j\geqslant 1.
	\end{align}

\begin{lemma}\label{laddersobolev2}
	For $k,j\geqslant 0$, the polynomials defined in \eqref{sobolev2} satisfy
	$$
	\frac{\partial^2}{\partial z \partial \bar{z}}[(1-z\bar{z})U_{k,j}(z,\bar{z})]\,=\,c_{k,j}\,Q_{k,j}^0(z,\bar{z}),
	$$
	with
	$$
	c_{k,j}\,=\,\left\{\begin{array}{rr} -1, & k=j=0,\\ 2, & k,j\geqslant 1, \end{array}\right.
	$$
	and 
	$$
	(1-z\bar{z})\frac{\partial^2}{\partial z \partial \bar{z}}Q_{k,j}^0(z,\bar{z})\,=\,d_{k,j}U_{k,j}(z,\bar{z}),
	$$
	with
	$$
	d_{k,j}\,=\,\left\{\begin{array}{ll} 0, & k=j=0,\\[10pt] \dfrac{1}{2}k\,j\,(k+1)\,(j+1), & k,j\geqslant 1. \end{array}\right.
	$$
\end{lemma}

\begin{proof}
	We compute,
	$$
	\frac{\partial}{\partial \bar{z}}[(1-z\bar{z})U_{0,0}(z,\bar{z})]\,=\,-z.
	$$
	Then,
	$$
	\frac{\partial^2}{\partial z \partial \bar{z}}[(1-z\bar{z})U_{0,0}(z,\bar{z})]\,=\,-1\,=\,-Q_{0,0}^0(z,\bar{z}).
	$$
	For $k,j\geqslant 1$, we compute
	$$
	\frac{\partial}{\partial \bar{z}}[(1-z\bar{z})U_{k,j}(z,\bar{z})]\,=\,(1-z\bar{z})\left\{(1-z\bar{z})\frac{\partial}{\partial \bar{z}} -2 z\right\} Q_{k-1,j-1}^2(z,\bar{z}).
	$$
	It follows from \eqref{ladderZ1} that
	$$
	\frac{\partial}{\partial \bar{z}}[(1-z\bar{z})U_{k,j}(z,\bar{z})]\,=\,-2(1-z\bar{z}) Q_{k,j-1}^1(z,\bar{z}).
	$$
	Then,
	$$
	\frac{\partial^2}{\partial z \partial \bar{z}}[(1-z\bar{z})U_{k,j}(z,\bar{z})]\,=\,-2\left\{(1-z\bar{z})\frac{\partial}{\partial z}-\bar{z} \right\}Q_{k,j-1}^1(z,\bar{z})\,=\,2\,Q_{k,j}^0(z,\bar{z}),
	$$
	where the last equality follows from \eqref{ladderZ2}.
	
	Moreover, using \eqref{ladderZ1} and \eqref{ladderZ2}, we obtain
	\begin{align*}
	(1-z\bar{z})\frac{\partial^2}{\partial z \partial \bar{z}}Q_{k,j}^0(z,\bar{z})&\,=\,\frac{1}{2}k\,j\,(k+1)\,(j+1)(1-z\bar{z})Q_{k-1,j-1}^2(z,\bar{z})\\[10pt]
	&\,=\,\frac{1}{2}k\,j\,(k+1)\,(j+1)\,U_{k,j}(z,\bar{z}).
	\end{align*}
\end{proof}

\begin{prop}
	The polynomials defined in \eqref{sobolev2} constitute a mutually orthogonal polynomial system with respect to $(\cdot,\cdot)_2$. Moreover,
	$$
	(U_{0,0},U_{0,0})_2\,=\,1, \quad (U_{k,j},U_{m,\ell})_2\,=\,\frac{4}{k+j+1}\,\delta_{k,m}\,\delta_{j,\ell}, \quad k,j\geqslant 1.
	$$
	 
\end{prop}

\begin{proof}
	By Lemma \eqref{laddersobolev2}, we have
	$$
	(U_{k,j},U_{m,\ell})_2\,=\,c_{k,j}^2\frac{1}{\pi}\int_{\mathbf{D}}Q_{k,j}^0(z,\bar{z})\,\overline{Q_{m,\ell}^0(z,\bar{z})}\,dz\,=\,c_{k,j}^2\,h_{k,j}^0\,\delta_{k,m}\,\delta_{j,\ell},
	$$
	where $h^0_{k,j}$ is given in \eqref{norm}. 
\end{proof}



\begin{thebibliography}{15}



\bibitem{DPP13}
A. M. Delgado, T. E. Pérez, M. A. Piñar, 
\textit{Sobolev-type orthogonal polynomials on the unit ball},
J. Approx. Theory \textbf{170} (2013), 94--106.


\bibitem{DFLPP16}
A. M. Delgado, L. Fernández, T. E. Pérez, M. A. Piñar, 
\textit{Sobolev orthogoanl polynomials on the unit ball via outward normal derivatives},
J. Math. Anal. Appl. \textit{440} (2016), no. 2, 716--740.


\bibitem{DX14}
C. F. Dunkl, Y. Xu,
    \textit{Orthogonal polynomials of several variables},
    2nd edition, Encyclopedia of Mathematics and its Applications, vol. 155, Cambridge Univ. Press, Cambridge (2014).
    
\bibitem{GM23}
 J. C. García-Ardila,  M. E. Marriaga,
 \textit{Sobolev orthogonality of polynomial solutions of second-order partial differential equations}, 
Comput. Appl. Math.\textbf{42} (2023), no.1, Paper No. 13, 44 pp. 

\bibitem{LX14}
H. Li, Y. Xu, 
	\textit{Spectral approximation on the unit ball}, 
	SIAM J. Numer. Anal. 52 (2014), no. 6, 2647--2675.
	
\bibitem{LPP21}
F. Lizarte, T. E. Pérez, M. A. Piñar,
\textit{The radial part of a class of Sobolev polynomials on the unit ball},
Numer. Algorithms \textit{87} (2021), no. 4, 1369--1389.

\bibitem{MX15}
F. Marcellán, Y. Xu, \textit{On Sobolev orthogonal polynomials}, Expo. Math. 33 (2015), no. 3, 308-352.

\bibitem{MPPR23}
M. E. Marriaga, T. E. Pérez, M. A. Piñar, M. J. Recarte,
\textit{Approximation via gradients on the ball. The Zernike case},
J. Comput. Appl. Math. \textbf{430} (2023), Paper No. 115258, 23.

\bibitem{MPR23}
M. E. Marriaga, T. E. Pérez, M. J. Recarte,
\textit{Simultaneous approximation via Laplacians on the unit ball},
Mediterr. J. Math. \textbf{20} (2023), no. 6, Paper No. 316, 22 pp.

 
\bibitem{NIST2010}
F. W. J. Olver, D. W. Lozier, R. F. Boisvert, C. W. Clark, (eds),
\textit{NIST Handbook ofMathematical
 Functions},
Cambridge University Press, Cambridge (2010).

\bibitem{PPX13}
T. E. Pérez, M. A. Pi{\~n}ar, Y. Xu, \textit{Weighted Sobolev orthogonal polynomials on the unit ball}, J. Approx. Theory, 171 (2013), 84–104.
    
\bibitem{PX09}
M. A. Pi\~nar, Y. Xu,
\textit{Orthogonal Polynomials and Partial Differential Equations on the Unit Ball}
Proc. Amer. Math. Soc. \textbf{137}, 2979-2987 (2009).

\bibitem{Sz78}
G. Szeg\H{o},
\textit{Orthogonal polynomials},
4th edition, vol.  23. Amer. Math. Soc. Colloq. Publ., Amer. Math. Soc., Providence RI, 1975.

\bibitem{T08}
A. Torre,
\textit{Generalized Zernike or disc polynomials: an application in quantum optics},
J. Comput. Appl. Math. \textbf{222} (2008), no. 2, 622--644.

\bibitem{VK1993}
N. Ja. Vilenkin, A. U. Klimyk, 
\textit{Representation of Lie groups and special functions. Vol. 2.}: Class I representations, special functions, and integral transforms. Translated from the Russian by V. A. Groza and A. A. Groza, Math. Appl. (Soviet Ser.), 74, Kluwer Academic Publishers Group, Dordrecht, 1993. xviii+607 pp.

\bibitem{W05}
A. Wünsche,
\textit{Generalized Zernike or disc polynomials},
J. Comput. Appl. Math. \textbf{174} (2005), no. 1, 135--163.

\bibitem{Xu06}
Y. Xu, 
\textit{A family of Sobolev orthogonal polynomials on the unit ball}, 
J. Approx. Theory 138 (2006), no. 2, 232-241.

\bibitem{Xu08}
Y. Xu, \textit{Sobolev orthogonal polynomials defined via gradient on the unit ball}, J. Approx. Theory, 152 (2008), 52--65.

\bibitem{Xu15}
Y. Xu,
\textit{Complex versus real orthogonal polynomials in two variables},
Integral Transforms Spec. Funct. \textbf{26} (2015), no. 2, 134--151.

\bibitem{Ze34}
Zernike F., \textit{Beugungstheorie des schneidenver-fahrens und seiner verbesserten form, der phasenkontrastmethode}, 
Physica Vol. 1, Issue 7-12 (1934), 689-704.

\end{thebibliography}
\end{document}